\documentclass[a4paper,10pt]{article}
\usepackage{amsmath, amssymb}
\usepackage{amsthm}
\usepackage{pdflscape}
\usepackage{multirow}
\usepackage{algorithm}

\usepackage{hyperref}

{
\theoremstyle{definition}
\newtheorem{definition}{Definition}[section]

}
\newtheorem{theorem}[definition]{Theorem}

\newtheorem{lemma}[definition]{Lemma}

\title{An algorithm for counting number of all (normal) fuzzy subgroups in $U_{6n}$}
\author{Marek Hy\v{c}ko\footnote{Email: marek.hycko@mat.savba.sk.\\ The research was supported by VEGA 2/0142/20 and APVV-20-0069 grants.}\\
Mathematical Institute Slovak Academy of Sciences,\\
\v{S}tef\'{a}nikova 49, SK-814 73 Bratislava}
\date{}

\begin{document}
\maketitle
\begin{abstract}
    Using already known resuls concerning the structure of (normal) subgroups of a $U_{6n}$ group we provide a dynamical programming algoritm for counting the number of all (normal) fuzzy subgroups of $U_{6n}$ with respect to M. T\u{a}rn\u{a}ceanu and L. Bentea equivalence relation.
\end{abstract}
\section{Preliminaries}
In this section we recall some notions which will be used in the paper. For the full coverage please refer to references.

Fuzzy groups were introduced by A. Rosenfeld (\cite{Rosen2001}).

\begin{definition}
Let $(G, e)$ be a group (with multiplicative operation) and $e$ be its neutral element. A fuzzy subset $\mu\in [0,1]^G$ of $G$ is called a \emph{fuzzy subgroup} (of $G$), if the following conditions are satisfied:
\begin{itemize}
	\item[(FG1)] $\mu(xy)\geq \min\{\mu(x), \mu(y)\}$,
	\item[(FG2)] $\mu(x^{-1})\geq \mu(x)$.
\end{itemize}
We call a fuzzy subgroup of $G$ \emph{normal}, if $\mu(xy) = \mu(yx)$, for any $x,y\in G$.
\end{definition}


The number of all (normal) fuzzy subgroups even for a trivial one-element group is infinite (since fuzzy levels are from the unit interval $[0,1]$). For counting fuzzy subgroups we will count the number of equivalence classes of the equivalence relation $\sim$ used by M. T\u{a}rn\u{a}uceanu and L. Bentea in \cite{TarBen2008}, which is for fuzzy subgroups $\mu$, $\nu$ of $G$ defined as follows:
\[ \mu \sim \nu \mbox{ if and only if } (\mu(x) > \mu (y) \Leftrightarrow \nu(x) > \nu(y), \mbox{ for all } x, y\in G).\]

It is worth to note that in history another kind of relation was used. Namely, by V. Murali and B. B. Makamba (\cite{MuMa2001-1}), which contained also condition that $\mu$ and $\nu$ have the equal support as fuzzy sets (moreover they considered only such fuzzy subgroups, where $\mu(e) = 1$). If we denote such relation as $\sim_2$ then \[
	\begin{split}
		|\{[\nu]_{\sim_2}: \nu \mbox{ is a fuzzy subgroup of } G\}| &= \\
		&2 |\{[\nu]_{\sim}: \nu \mbox{ is a fuzzy subgroup of } G\}| - 1.
	\end{split}
\] (By \cite{Volf2004} fuzzy subgroups are characterized as chains of subgroups ending in $G$, which are mapped to sequences of membership-degree levels $1> t_1>\ldots> t_k\geq 0$. Relation $\sim_1$ distinguishes just relations between $t_i$'s, but $\sim_2$ distinguishes whether the last element $t_k$ if present is non-zero or zero, thus each non-trivial chain ending in $G$ gives rise to two classes with respect to $\sim_2$ relation.)

We will consider groups $U_{6n}$ that are for $n\in\mathbb{N}$ groups with presentation
\[
   U_{6n} = \langle a, b \mathrel{|} a^{2n} = b^3 = 1, bab = a\rangle.
\]

The structure of them were studied in \cite{SheAsh2019}. We summarize main findings, which are relevant to this study. Each element is a sequence of letters $a$ and $b$, which can be represented in a canonical way, i.e., $a^s b^t$, where $s = 0, 1, 2, \ldots 2n - 1$ and $t = 0, 1, 2$. (This follows from the fact that it holds $ba = ab^2$ and $b^2a = ab$.) We note, that this representation is equivalent to $a^s b^t$, for $t = 1, 2, \ldots, 2n$, since $a^0 = a^{2n} = 1$. Throughout the paper we will use both such representations and it should be clear from the context, which one is used. 

Moreover, for powers of canonical elements it holds:

$(a^s b^t)^k = a^{sk \mod 2n}$, if $t = 0$ 

$(a^s b^t)^k =a^{sk \mod 2n}b^{tk\mod 3}$, if $s$ is even

$(a^s b^t)^k = a^{sk \mod 2n}$, if $t \neq 0$ and $s$ is odd and $k$ is even

$(a^s b^t)^k =a^{sk \mod 2n}b^t$, if $t \neq 0$ and $s$ is odd and $k$ is odd

Each subgroup of $U_{6n}$ is of the form $\langle a^r, b\rangle$, or $\langle a^r\rangle$, or $\langle a^r b\rangle$, or $\langle a^r b^2\rangle$, for $r\mathrel{|}2n$.

\begin{lemma}\label{lem:subg}  
Let $G_1$ and $G_2$ be subgroups of $U_{6n}$. Then the following holds 

\begin{enumerate}

\item $\langle a^t, b\rangle$ is not subgroup of $\langle a^{t'}\rangle$ for any $t, t'\in \mathbb{N}$, $t, t' | 2n$.

\item $t_1, t_2 | 2n$ and if $t_1 \nmid t_2$, then $\langle a^{t_2}\rangle$ is not subgroup of $\langle a^{t_1}\rangle$ and 
$\langle a^{t_2}, b\rangle$ is not subgroup of $\langle a^{t_1}, b\rangle$.

\item $t_1, t_2 | 2n$ and if $t_1 | t_2$ and for $r:= t_2 / t_1$,  $(a^{u_1}b^{v_1})^r \neq a^{u_2}b^{v2}$, then $\langle a^{t_2}b^s\rangle$ is not subgroup of $\langle a^{t_1}b^s\rangle$, for $s =0,1,2$.
\end{enumerate}
\end{lemma}
\begin{proof}
    We consider generators in a canonical form, i.e., exponents of an element $a$ divide $2n$.
    No $\langle a^t, b\rangle$ is a subgroup of $\langle a^s\rangle$, which covers the first part. 
    
    If $\langle a^t, b\rangle$ is a subgroup of $\langle a^s, b\rangle$  (or $\langle a^t\rangle$ is a subgroup of $\langle a^s\rangle$), then $a^t\in \langle a^s\rangle$, which is possible only if $s|t$. This covers the second part. 

    An element $a^{t_1} b^{s_1}$ is an element of $\langle a^{t_2} b^{s_2}\rangle$, if and only if $t_2 | t_1$ and for $r:=t_1 / t_2$ it holds $(a^{t_2} b^{s_2})^r = a^{t_1} b^{s_1}$. Which covers the last part.
\end{proof}

\begin{theorem}\label{thm:allsubgroups}
    Let $n\in \mathbb{N}$. 
    Then the set of all subgroups of $U_{6n}$ 
    consists of the following mutually exclusive elements:
    \begin{enumerate}
        \item $\langle a^t\rangle$ and $\langle a^t, b\rangle$, for all $t\in \{1, 2, \ldots, 2n\}$, $t|2n$;
        \item $\langle a^t b\rangle$ and $\langle a^t b^2\rangle$, for all $t\in \{1, 2, \ldots, 2n\}$, $t|2n$ such that $t$ is odd or ($t$ is even and $\frac{2n}{t}\equiv 0\mod 3$).
    \end{enumerate}
\end{theorem}

\begin{proof}
    The following holds:

    Each subgroup $\langle a^t\rangle$ is cyclic group of order $\frac{2n}{t}$, thus it is isomorphic to $\mathbb{Z}_{\frac{2n}{t}}$.

    Each subgroup $\langle a^t, b \rangle$ for $t$ even is isomorphic to $\mathbb{Z}_3\times \mathbb{Z}_{\frac{2n}{t}}$ and it is cyclic, iff $(3, \frac{2n}{t}) = 1$. For $t$ odd it is not a cyclic group. In both cases orders are $3\cdot\frac{2n}{t}$.

    So if we consider $t_1, t_2 \in \{1, 2, \ldots, 2n\}$ such that $t_1, t_2 | 2n$ and $t_1\neq t_2$, then it is immediately to see that $\langle a^{t_1}\rangle \neq \langle a^{t_2}\rangle$ (orders of subgroups differ). 

    Similarly, $\langle a^{t_1}, b\rangle \neq \langle a^{t_2}, b\rangle$. Moreover, since $\langle a^{t_1}, b\rangle$ contains powers of $b$, we have that $\langle a^{t_1}, b\rangle \neq \langle a^{t}\rangle$, for any $t | 2n$. 

    
    If $t$ is odd, $t|2n$ and $s = 1$ or $s = 2$, then $\langle a^t b^s\rangle = \{ a^{tk} b^{s(k\mod 2)}: k = 1, 2, \ldots, 2n/t \}$. Is is a cyclic group of order $\frac{2n}{t}$, which is certainly not in the form of $\langle a^{t'}, b\rangle$ for any $t'\in \mathbb{N}$, neither it is equal to $\langle a^{t'}\rangle$ since it contains elemets with powers of $b$. Moreover, we have that $\langle a^t b\rangle\neq \langle a^t b^2\rangle$ (powers of $b$s in existing elements differ, the former contains just $b$, the latter just $b^2$).


    If $t$ is even, $t|2n$ and $s = 1$ or $s=2$ then $\langle a^t b^s\rangle = \{a^{tk}b^{sk}: k = 1, \ldots, 2n/t \}$. It is again a cyclic group of order $\frac{2n}{t}$. In case that $d:= s(2n/t)\not\equiv 0 \mod 3$, then $\langle a^t b^s\rangle$ contains element  $a^{s(2n)}b^d= b^{d\mod 3}$, $d\mod 3\neq 0$. Combining it with existing powers lead to the fact that  $\langle a^t b^s\rangle =  \langle a^t, b\rangle$, which was listed in the first part. If $3|\frac{2n}{t}$, then it is obvious that $\langle a^t b\rangle \neq \langle a^t b^2\rangle$ and certainly both are not equal to $\langle a^t\rangle$.
\end{proof}

\begin{theorem}\label{thm:allnormal}
    Let $n \in \mathbb{N}$ 
    Then the set of all normal subgroups of $U_{6n}$ consists of the following mutually distinct elements, $t\in \{1, 2, \ldots, 2n\}$:
    \begin{itemize}     
        \item $\langle a^t\rangle$, for all $t|2n$, $t$ even;
        \item $\langle a^t, b\rangle$, for all $t|2n$.
    \end{itemize}
\end{theorem}

\begin{proof}
    Distinction of elements is obvious. According to previous theorem we need to check 4 cases. In general we need to show, than $g^{-1}Hg = H$, for any candidate $H$ for a normal group and any element of $G = U_{6n} = \{a^r b^s: r = 1, \ldots, 2n, s = 0, 1, 2\}$. To shorten display we will omit $\mod 2n$ in exponents of $a$ and $\mod 3$ in exponents $b$. If $g = a^rb^s$, then $g^{-1} = b^{3-s}a^{2n-r}$.

    \noindent Case $H = \langle a^t\rangle$, $t|2n$:

    $H = \{a^{tk\mod 2n}: k = 1, 2, \ldots, 2n/t\}\cong \mathbb{Z}_{2n/t}$

    $g^{-1}Hg = b^{3-s}a^{2n-r}a^{tk} a^r b^s = b^{3-s}a^{tk}b^s = \left\{\begin{matrix}a^{tk},& tk \mbox{ even}\\
        a^{tk}b^{2(3-s) + s},& tk\mbox{ odd}\end{matrix}\right. = \\  = \left\{\begin{matrix}a^{tk},& tk \mbox{ even}\\
            a^{tk}b^{-s},& tk\mbox{ odd}\end{matrix}\right.$. We see that for odd elements $t$ we get that $H\neq g^{-1}Hg$, i.e., $H$ is not a normal group.
    
    \noindent Case $H = \langle a^t b\rangle$, $t|2n$:

    $H = \{a^{tk} b^{k\mod 2}: k = 1, 2, \ldots, 2n/t\}\cong \mathbb{Z}_{2n/t}$

    $g^{-1}Hg = b^{3-s}a^{2n-r}a^{tk} b^{k\mod 2} a^r b^s = \\  = 
    \left\{\begin{matrix}
        a^{2n - r + tk}b^{(3-s) + (k\mod 2)} a^r b^s,& -r + tk \mbox{ even}\\
        a^{2n - r + tk}b^{2(3-s) + (k\mod 2)} a^r b^s,& -r + tk \mbox{ odd}\\
    \end{matrix}\right. = \\ 
    = \left\{\begin{matrix}
        a^{tk} b^{k\mod 2},& r\mbox{ even}, tk \mbox{ even}\\
        a^{tk} b^{2(k\mod 2)-s},& r\mbox{ odd}, tk \mbox{ odd}\\
        a^{tk} b^{(k\mod 2) -s},& r\mbox{ even}, tk \mbox{ odd}\\
        a^{tk} b^{2(k\mod 2)},& r\mbox{ odd}, tk \mbox{ even}\\
    \end{matrix}\right.$. We see that there is no normal group of this form.

    \noindent Case $H = \langle a^t b^2\rangle$, $t|2n$:

    $H = \{a^{tk} b^{2(k\mod 2)}: k = 1, 2, \ldots, 2n/t\}\cong \mathbb{Z}_{2n/t}$

    $g^{-1}Hg = b^{3-s}a^{2n-r}a^{tk} b^{2(k\mod 2)} a^r b^s = \\  = 
    \left\{\begin{matrix}
        a^{2n - r + tk}b^{(3-s) + 2(k\mod 2)} a^r b^s,& -r + tk \mbox{ even}\\
        a^{2n - r + tk}b^{2(3-s) + 2(k\mod 2)} a^r b^s,& -r + tk \mbox{ odd}\\
    \end{matrix}\right. = \\ 
    = \left\{\begin{matrix}
        a^{tk} b^{2(k\mod 2)},& r\mbox{ even}, tk \mbox{ even}\\
        a^{tk} b^{(k\mod 2)-s},& r\mbox{ odd}, tk \mbox{ odd}\\
        a^{tk} b^{2(k\mod 2) -s},& r\mbox{ even}, tk \mbox{ odd}\\
        a^{tk} b^{(k\mod 2)},& r\mbox{ odd}, tk \mbox{ even}\\
    \end{matrix}\right.$. We see that there is no normal group of this form.
    
    \noindent Case $H = \langle a^t,  b\rangle$, $t|2n$:
    
    $H = \{a^{tk} b^u: k = 1, 2, \ldots, 2n/t, u = 0, 1, 2\}$

    $g^{-1}Hg = b^{3-s}a^{2n-r}a^{tk} b^u a^r b^s = \\
     = \left\{\begin{matrix}
        a^{2n - r + tk}b^{(3-s) + u} a^r b^s,& -r + tk \mbox{ even}\\
        a^{2n - r + tk}b^{2(3-s) + u} a^r b^s,& -r + tk \mbox{ odd}\\
    \end{matrix}\right. = \\ 
    = \left\{\begin{matrix}
        a^{tk} b^{u},& r\mbox{ even}, tk \mbox{ even}\\
        a^{tk} b^{2u - s},& r\mbox{ odd}, tk \mbox{ odd}\\
        a^{tk} b^{u - s},& r\mbox{ even}, tk \mbox{ odd}\\
        a^{tk} b^{2u},& r\mbox{ odd}, tk \mbox{ even}\\
    \end{matrix}\right.
    $.

    Since each mapping depended on the variable $u$ (computed modulo 3), functions $u$, $2u-s$, $u - s$ and $2u$ are bijections from $\{0,1,2\}$ onto $\{0,1,2\}$ for any fixed $s\in \{0,1,2\}$, we get that each $\langle a^t, b\rangle$, $t|2n$ is a normal group.
\end{proof}

The previous theorems allows us to characterize the sets of all (normal) subgroups of $U_{6n}$.
Let $n = 2^{p_2}3^{p_3}p_1^{k_1}\ldots p_l^{k_l}$ be a (extended) canonical decomposition of $n$ ($p_2$ or $p_3$ can equal to $0$), where $p_i\neq 2, 3$ are distinct odd primes, $i = 1, 2, \ldots, l$. Then all subgroups of $U_{6n}$ are 
\begin{itemize}
    \item $\langle a^k\rangle$, $\langle a^k, b\rangle$, for $k = 2^{s_2} 3^{s_3} p_1^{t_1} \ldots p_l^{t_l}$, where $0\leq s_2\leq p_2 + 1$, $0\leq s_3\leq p_3$ and for $i = 1, 2, \ldots, l$, $0\leq t_i\leq k_i$.

    \item $\langle a^k b\rangle$, $\langle a^k b^2\rangle$, for $k = 2^{s_2} 3^{s_3} p_1^{t_1} \ldots p_l^{t_l}$, where
    \begin{itemize}
        \item for $s_2 = 0$: $0\leq s_3\leq p_3$ and for $i = 1, 2, \ldots, l$, $0\leq t_i\leq k_i$;  
        \item for $1\leq s_2\leq  p_2 + 1$: $0\leq s_3\leq p_3 - 1$  and for $i = 1, 2, \ldots, l$, $0\leq t_i\leq k_i$.
    \end{itemize}
\end{itemize}

All normal subgroups of $U_{6n}$ are 
\begin{itemize}
    \item $\langle a^k\rangle$, where $k = 2^{s_2} 3^{s_3} p_1^{t_1} \ldots p_l^{t_l}$, where $1\leq s_2\leq p_2 + 1$, $0\leq s_3\leq p_3$ and for $i = 1, 2, \ldots, l$, $0\leq t_i\leq k_i$.
    
    \item $\langle a^k, b\rangle$where $k = 2^{s_2} 3^{s_3} p_1^{t_1} \ldots p_l^{t_l}$, where $0\leq s_2\leq p_2 + 1$, $0\leq s_3\leq p_3$ and for $i = 1, 2, \ldots l$, $0\leq t_i\leq k_i$.
\end{itemize}

We denote the set of all subgroups of a group $G$ by $SG(G)$. We are able to define strict partial order $G_1\prec G_2$, iff $G_1$ is a subgroup of $G_2$ and $G_1\neq G_2$. By $SG^*(G)$ we denote the set of all subgroups except the trivial one-element subgroup. 

To compute the number of chains of subgroups  with length $n$ ending in $G$ (the whole group) and not containing the trivial one-element subgroup $\{e\}$, the so called \emph{proper $n$-chains in $G$}, we use dynamic programming. We denote $L(i, G_j)$ the number of chains of length $i$ starting in $G_j$ and ending in the whole group $G$, for any subgroup $G_j$ of $G$. Then we have recurrence relation \[L(i+1, G_j) = \sum_{G_k\in SG^*(G)} \{L(i, G_k): G_j \prec G_k\},\] with boundary conditions \[L(1, G) = 1 
\mbox{ and }L(1, G_k) = 0\mbox{ for all }G_k\in SG^*(G)\setminus\{G\}.\]

Similarly for normal subgroups we denote the set of all normal subgroups of $G$ by $SG_n(G)$, and $SG_n^*(G) = SG_n(G) \setminus \{e\}$. The strict relation $G_1 \prec_n G_2$ is defined as $G_1$ is a proper normal subgroup of $G_2$ and for the number of chains of normal subgroups of length $n$ starting in $G_j$ ending in $G$ and not containing trivial one-element subgroup we obtain recurrence \[L_n(i+1, G_j) = \sum_{G_k\in SG_n^*(G)} \{L_n(i, G_k): G_j\prec_n G_k\},\] with boundary conditions \[L_n(1, G) = 1\mbox{ and }L_n(1, G_i) = 0\mbox{ for any }G_i\in SG_n^*\setminus \{G\}.\]

In both recurrences sums of empty sets are equal to $0$.

We denote the counts of (normal) proper $n$-chains in $G$ by $L^G_n$ (${}^nL^G_n$)  for $n\geq 1$, respectively. Then \[
	L^G_n = \sum_{G_k\in SG^*(G)} L(n, G_k)\mbox{ and } {}^nL^G_n = \sum_{G_k\in SG_n^*(G)} L_n(n, G_k).\]

Then by \cite{Volf2004} we have that the number of all (normal) fuzzy subgroups of $G$ (with respect to relation $\sim$) is equal to $N_F(G) = 2\sum_{n\geq 1} L^G_n$ ($N_{NF}(G) = 2\sum_{n\geq 1} {}^nL^G_n$). (This is justified by the fact that each chain of subgroups ending in $G$ and not containing trivial one-element subgroup gives rise to precisely two chains of subgroups ending in $G$ - itself and the one with appended one-element subgroup.)

To summarize previous reasoning we provide Algorithm~\ref{alg:numchains} for computing the number of $n$-chains and number of fuzzy subgroups of $G$. The set $SG^*(G)$ is obtained from Theorem~\ref{thm:allsubgroups}, the precomputation phase is governed by Lemma~\ref{lem:subg}.

Similar algorithm can be obtained for counting normal fuzzy subgroups and normal $n$-chains (chains of length $n$ of normal subgroups ending in G not containing trivial one-element subgroup.) The set $nSG^*$ is obtained from Theorem~\ref{thm:allnormal} and $\prec_n$ relation is governed by Lemma~\ref{lem:subg}.

\section*{Algorithms}

\begin{algorithm}[H]\label{alg:numchains}
    \caption{Algorithm for computing $L^G_n$ and $N_F(G)$}
Input: $SG^*(G)$

Output: $L^G_n$ for $n\geq 1$ and $N_F(G)$

    1. Precomputation of relation $\prec$

    2. Initial conditions: 
    
    \hspace*{1cm} $L^G_1 = L(1,G) = 1$ and $L(1, H) = 0$, for $H\in SG^*(G)\setminus \{G\}$.

    3. $s = 1$, $k = 1$

    4. \textbf{while} $s > 0$

    \hspace*{1cm} $tmps = 0$

    \hspace*{1cm} \textbf{for each} $H \in SG^*(G)$ do

    \hspace*{2cm} $L(k+1, H) = \sum_{H\prec S} L(k, S)$
            
    \hspace*{2cm} $tmps = tmps + L(k+1, H)$

    \hspace*{1cm} $L^G_{k+1} = tmps$ 
    
    \hspace*{1cm} $s = tmps$ 
        
    \hspace*{1cm} $k = k + 1$
        
    5. Return 1: 
    
    \hspace*{1cm} $L^G_n$ 
    , for 
    $1\leq n\leq k-1$
    
    \hspace*{1cm} $L^G_n = 0$, for $n\geq k$.

    6. Return 2: 
    
    \hspace*{1cm} $N_F(G) = 2\sum_{n\geq 1} L^G_n = 2\sum_{n = 1}^{k-1} L^G_n$.
\end{algorithm}

\begin{algorithm}[H]\label{alg:normnumchains}
    \caption{Algorithm for computing ${}^nL^G_n$ and $N_{NF}(G)$}
Input: $SG^*_n(G)$

Output: ${}^nL^G_n$ for $n\geq 1$ and $N_{NF}(G)$

    1. Precomputation of relation $\prec_n$

    2. Initial conditions: 
    
    \hspace*{1cm} ${}^nL^G_1 = Ln(1,G) = 1$ and $Ln(1, H) = 0$, for $H\in SG^*_n(G)\setminus \{G\}$.

    3. $s = 1$, $k = 1$

    4. \textbf{while} $s > 0$

    \hspace*{1cm} $tmps = 0$

    \hspace*{1cm} \textbf{for each} $H \in SG^*_n(G)$ do

    \hspace*{2cm} $Ln(k+1, H) = \sum \{Ln(k, S): S\in SG^*_n(G), H\prec_n S\}$
            
    \hspace*{2cm} $tmps = tmps + Ln(k+1, H)$

    \hspace*{1cm} ${}^nL^G_{k+1} = tmps$ 
    
    \hspace*{1cm} $s = tmps$ 
        
    \hspace*{1cm} $k = k + 1$
        
    5. Return 1: 
    
    \hspace*{1cm} ${}^nL^G_n$ 
    , for 
    $1\leq n\leq k-1$
    
    \hspace*{1cm} ${}^nL^G_n = 0$, for $n\geq k$.

    6. Return 2: 
    
    \hspace*{1cm} $N_{NF}(G) = 2\sum_{n\geq 1} {}^nL^G_n = 2\sum_{n = 1}^{k-1} {}^nL^G_n$.
\end{algorithm}

\begin{thebibliography}{X}
	\bibitem{MuMa2001-1} V. Murali, B. B. Makamba, \textit{On an equivalence of fuzzy subgroups} I, Fuzzy Sets and Systems \textbf{123}(2)(2001), 259--264. DOI: \href{https://doi.org/10.1016/S0165-0114(00)00098-1}{10.1016/S0165-0114(00)00098-1}.

    \bibitem{Rosen2001} A. Rosenfeld, \textit{Fuzzy subgroups}, Journal of Mathematical Analysis and Applications \textbf{35} (1971), 512--517. DOI: \href{https://doi.org/10.1016/0022-247X(71)90199-5}{10.1016/0022-247X(71)90199-5}.

	\bibitem{SheAsh2019} H. B. Shelash, A. R. Ashrafi, \textit{Computing maximal and minimal subgroups with respect to a given property in certain finite groups}, Quasigroups and Related Systems \textbf{27} (2019), 133--146. Online: \href{https://www.math.md/files/qrs/v27-n1/v27-n1-(pp133-146).pdf}{[pdf]}.

	\bibitem{TarBen2008} M. T\u{a}rn\u{a}uceanu, L. Bentea, \textit{On the number of fuzzy subgroups of finite abelian groups}, Fuzzy Sets and Systems \textbf{159}(9)(2008), 1084--1096. DOI: \href{https://doi.org/10.1016/j.fss.2007.11.014}{10.1016/j.fss.2007.11.014}.

	\bibitem{Volf2004} A. C. Volf, \textit{Fuzzy subgroups and chains of subgroups}, Ia\v{s}i Series Fuzzy Sets and Artificial Intelligence \textbf{10}(3)(2004), 87--98. Online: \href{http://mss.academiaromana-is.ro/2004_2/87-98_Aurelian%20Claudiu%20VOLF.pdf}{[pdf]}.

\end{thebibliography}
\end{document}